\documentclass[leqno,a4paper,12pt]{amsart}

\usepackage[T1]{fontenc}
\usepackage{amsmath,amsthm,amssymb,amsfonts,enumerate}
\usepackage{xcolor}
\usepackage[colorlinks,citecolor=blue,linkcolor=blue,urlcolor=blue,
            hypertexnames=false]{hyperref}
\allowdisplaybreaks[4]
\newtheorem{theorem}{Theorem}[section]
\newtheorem{lemma}[theorem]{Lemma}
\newtheorem{proposition}[theorem]{Proposition}
\newtheorem{corollary}[theorem]{Corollary}
\newtheorem{remark}[theorem]{Remark}
\newtheorem{definition}[theorem]{Definition}

\newcommand{\essinf}{\operatorname*{ess\,inf}}
\newcommand{\BLO}{\mathrm{BLO}}
\newcommand{\BMO}{\mathrm{BMO}}
\newcommand{\R}{\mathbb{R}}
\newcommand{\norm}[1]{\left\lVert #1\right\rVert}
\newcommand{\abs}[1]{\left\lvert #1\right\rvert}

\begin{document}

\title[Sharp One-Sided John--Nirenberg Inequality for BLO]
{Sharp constants in the One-sided John--Nirenberg inequality
for functions of bounded lower oscillation}


\author{Chao Zhang}
 \address{{Department of Mathematics, Zhejiang University of Science and
 Technology,\linebreak Hangzhou 310023, P. R. China}}
 \email{zaoyangzhangchao@163.com}
\thanks{The  author is supported  by the National Natural Science Foundation of China (Grant No.  11971431)}

\date{}

\begin{abstract}
{We determine the sharp constants in the one-sided John--Nirenberg inequality
for functions of bounded lower oscillation on an interval.  More precisely,
if $f\in\BLO(I_0)$, then
\[
 \sup_{I\subseteq I_0}\frac{1}{|I|}
 \left|\left\{x\in I:
 f(x)-\essinf_I f>\lambda\right\}\right|
 \le c_1\exp\left(-\frac{c_2\lambda}{\norm{f}_{\BLO(I_0)}}\right),
 \qquad \lambda>0,
\]
with the usual interpretation when \(\|f\|_{\mathrm{BLO}(I_0)}=0\).
The sharp constants are $c_1^*=e$ and $c_2^*=1$.  We give a proof based
on the Riesz rising sun lemma and an independent Bellman function proof.
Finally, we present several consequences of the sharp estimate.}
\end{abstract}
\maketitle

\section{Introduction}\label{sec:introduction}

The space of functions of bounded mean oscillation $(\BMO)$ was introduced by John and Nirenberg~\cite{JN61}, and has played an important role in analysis and probability. A real-valued locally integrable function $f$ defined on $\R^d$ is said to be in $\BMO(\R^d)$,
 if
\begin{equation}\label{eq:BMO-def}
\sup_{Q\subseteq \R^d} \frac{1}{|Q|} \int_Q |f(x) - f_Q| \,dx < \infty,
\end{equation}
where the supremum is over all cubes $Q$ in $\R^d$ with edges parallel to the coordinate
axes, $|Q|$ denotes the volume of $Q$ and
$\displaystyle
f_Q = \frac{1}{|Q|} \int_Q f(x)\,dx
$
is the mean of $f$ over $Q$.   The  bounded lower oscillation ($\BLO$) class first appeared in the paper of Coifman and Rochberg~\cite{CR80}.
A function $f$ belongs to $\BLO(\R^d)$,  if the term $f_Q$ in
\eqref{eq:BMO-def} is replaced by $\displaystyle \essinf_Q f$, the essential infimum of $f$ over cubes $Q$.
That is, $f \in \BLO(\R^d)$ if
\begin{equation}\label{eq:BLO-def}
\sup_Q \left( f_Q - \essinf_Q f \right) < \infty.
\end{equation}
The suprema in \eqref{eq:BMO-def} and \eqref{eq:BLO-def} are denoted by
$\|f\|_{\BMO(\R^d)}$ and $\|f\|_{\BLO(\R^d)}$. In~\cite{CR80}, Coifman and Rochberg proved the decomposition $\BMO = \BLO - \BLO$.
Unlike $\BMO$, the class $\BLO$ is not a linear space, because it is not stable under multiplication by negative numbers.
Furthermore, $\BLO \cap (-\BLO) = L^\infty$, and $\BMO$ is the smallest linear space containing the $\BLO$ class. In~\cite{LYZ}, Liang, Yang and Zhang gave some characterizations of $\BLO$ spaces by heat semigroups.

In this paper, we work mainly in the one-dimensional setting, where cubes are replaced by intervals
$I \subset \R$.
For a fixed interval $I_0$, we  consider the localized seminorms
\[
\|f\|_{\BMO(I_0)} = \sup_{I \subseteq I_0} \frac{1}{|I|} \int_I |f(x) - f_I| \,dx, \qquad
\|f\|_{\BLO(I_0)} = \sup_{I \subseteq I_0} \bigl( f_I - \essinf_I f \bigr).
\]

Functions of bounded mean oscillation satisfy the classical John--Nirenberg exponential
estimate~\cite{JN61}, i.e.,  there exist constants $c_1, c_2 > 0$ such that
\[
\sup_{I \subseteq I_0}\frac{1}{|I|} \bigl| \{x \in I : |f(x) - f_I| > \lambda\} \bigr|
\le c_1 e^{-c_2 \lambda / \|f\|_{\BMO(I_0)}}, \qquad \lambda > 0.
\]
{With this inequality in hand, the $\BMO$ seminorm defined above is equivalent to the $\BMO_p$ seminorm defined by $\displaystyle \norm{f}_{\BMO_p(I_0)}=\sup_{I \subseteq I_0} \bigl(\frac{1}{|I|} \int_I |f(x) - f_I|^p \,dx\bigr)^{1/p}$ for $1\le p<\infty$.}
{Using the Bellman function method, Vasyunin and
Volberg~\cite{VV14} determined the sharp constants for the $L^2$-based
$\BMO$ seminorm: $c_1^*=e^2/4\approx1.847$ and $c_2^*=1$ for big $\lambda$.
For the $L^1$-based $\BMO$ seminorm, Lerner~\cite{Lerner13} obtained the
sharp prefactor $c_1^*=e^{4/e}/2$ using the Riesz rising sun lemma, while
Korenovskii~\cite{Kore} obtained the sharp exponential rate $c_2^*=2/e$.
Slavin and Vasyunin~\cite{SV11} proved the sharp integral-form results.  In 2013,
again using the Bellman function method, Os\c{e}kowski~\cite{Osekowski}
proved that there are absolute constants $c_1=c_2=1$ such that}
\begin{equation*}\label{eq:jn-general}
\sup_{I \subseteq I_0} \frac{1}{|I|} {|\{x\in I : \abs{f(x) -  f_{I}} > \lambda \}|}
\leq c_1 e^{-{{c_2\lambda}/{\|f\|_{\mathrm{BLO}(I_0)}}}}, \qquad \lambda > 0,
\end{equation*}
{for all $f\in\BLO(I_0)$ and all bounded intervals $I_0$, and both
constants are sharp.  The Bellman function method is a powerful tool for
best-constant problems involving $\BMO$ and $\BLO$ functions.}

{Two earlier one-dimensional results are particularly relevant to the
applications considered below.  Passarelli di Napoli~\cite{Passarelli96}
obtained sharp relations between $\BLO_p$ and $\BLO$ quantities and a precise
$A_1$ estimate for exponentials of $\BLO$ functions.  More recently,
Popoli~\cite{Popoli25} identified the exact one-dimensional small-seminorm
threshold for the implication $\exp(f)\in A_1$.  The focus of the present
paper is different: we determine the optimal prefactor and exponential rate
in the weak, one-sided distribution inequality \eqref{eq:JN-BLO}, under the
convention stated below.  The exponential-integrability and moment estimates
in Section~\ref{sec:applications} are recorded only as consequences of that
distribution inequality and are not presented as new sharp $A_1$ or
$\BLO_p$ theorems.}

For BLO functions,  a more appropriate John--Nirenberg inequality is the one-sided John--Nirenberg inequality. This inequality was first established by Wang, Zhou and Teng~\cite{WZT18}: there exist constants $c_1, c_2 > 0$ such that
\begin{equation}\label{eq:JN-BLO}
\sup_{I \subseteq I_0} \frac{1}{|I|}\left|\left\{x \in I : f(x) - \essinf_{y\in I} f(y) > \lambda \right\}\right|
\le c_1 \, e^{-c_2 \lambda / \|f\|_{\BLO(I_0)}} , \qquad \lambda > 0.
\end{equation}
The purpose of this paper is to determine the \emph{sharp} constants in this inequality. Because $c_1$ and $c_2$ can be traded against one another, we use the following
standard convention.  The sharp exponential rate $c_2^*$ is the supremum of
all $c_2$ for which \eqref{eq:JN-BLO} holds with some finite $c_1$. Once
$c_2=c_2^*$ is fixed, $c_1^*$ is the smallest admissible prefactor.

The following is the main result.

\begin{theorem}\label{thm:main}
For the one-dimensional inequality \eqref{eq:JN-BLO},
\[
 c_1^*=e, \qquad c_2^*=1.
\]
Equivalently, for every bounded interval $I_0$, every
$f\in\BLO(I_0)$, every subinterval $I\subseteq I_0$, and every $\lambda>0$,
\begin{equation}\label{eq:main-bound}
 \frac1{|I|}
 \left|\left\{x\in I:
 f(x)-\essinf_I f>\lambda\right\}\right|
 \le
 \min\left\{1,
 \exp\left(1-\frac{\lambda}{\norm{f}_{\BLO(I_0)}}\right)
 \right\},
\end{equation}with the usual interpretation when \(\|f\|_{\mathrm{BLO}(I_0)}=0\).
Both constants are sharp.
\end{theorem}
{The proof is based on a stopping-time recursion.  After normalizing
$\essinf_I f=0$ and $\norm{f}_{\BLO(I_0)}\le1$, we apply the Riesz rising sun
lemma (see Lemma \ref{lem:sunrise}) to $f$ at height $1+\delta$.  On each stopping interval $J$, the BLO
condition implies
\[
 \essinf_J f\ge\delta.
\]
Subtracting $\delta$ therefore produces another admissible normalized
function on $J$.  This gives the recursion
\[
 \Phi(\lambda)\le\frac1{1+\delta}\Phi(\lambda-\delta)
\]
for the extremal distribution function $\Phi$.  Iteration followed by
$\delta\downarrow0$ yields $\Phi(\lambda)\le e^{1-\lambda}$.}

We also give an independent Bellman function proof in
Section~\ref{sec:bellman}.  Besides avoiding the rising-sun decomposition, that
argument retains the local mean excess.  More precisely, when
$\norm{f}_{\BLO(I_0)}>0$, if
\[
 x_I=\frac{f_I-\essinf_I f}{\norm{f}_{\BLO(I_0)}}
 \quad\text{and}\quad
 L=\frac{\lambda}{\norm{f}_{\BLO(I_0)}},
\]
then the Bellman argument gives the stronger estimate
\[
 \frac1{|I|}\left|\{f-\essinf_I f>\lambda\}\right|
 \le
 \begin{cases}
  1, & 0<L\le x_I,\\
  x_I/L, & x_I<L\le1,\\
  x_Ie^{1-L}, & L\ge1.
 \end{cases}
\]
Since $0\le x_I\le1$, this immediately implies \eqref{eq:main-bound}.

{For comparison, the classical one-dimensional two-sided John--Nirenberg
inequality for $\BMO$, with the seminorm displayed above, has sharp
constants $c_1^* = e^{4/e}/2$ and $c_2^*=2/e$; see \cite{Kore, Lerner13}.  The larger exponential rate obtained
here reflects the one-sided geometry of $\BLO$.  On each rising-sun interval of
average $1+\delta$, the $\BLO$ condition gives the pointwise lower bound
$\essinf f\ge\delta$, so subtraction of $\delta$ preserves nonnegativity and
returns the function to the same normalized admissible class.  The two-sided
$\BMO$ condition gives no analogous lower bound: positive and negative
oscillations must be controlled simultaneously, and the resulting sharp
recursion is different.}

The present argument is tied to the interval geometry of the one-dimensional
problem.  For cubes in $\mathbb{R}^d$, a standard Calder\'on--Zygmund
decomposition controls the stopping averages only up to a dimensional factor,
which destroys the exact renormalization used in the recursion.  There are
multidimensional rising-sun decompositions with exact averages for rectangles;
see Korenovskyy, Lerner, and Stokolos~\cite{KLS05}.  Such rectangles need not
be cubes, however, so the usual cube-based $\BLO$ seminorm does not directly
control all stopping sets.  The rectangular result suggests a possible route
for strong, rectangle-based $\BLO$, but the sharp higher-dimensional
cube-based problem requires additional ideas.

{This article is organized as follows.  Section~\ref{sec:ProofSharp}
contains two independent proofs of the sharp one-sided John--Nirenberg
inequality for $\BLO$ functions: one based on the Riesz rising sun lemma and
one based on the Bellman function method.  Section~\ref{sec:sharpness} proves
sharpness using extremizing examples.  Section~\ref{sec:applications} records
several consequences of the sharp estimate.}

\section{Proof of the sharp inequality}\label{sec:ProofSharp}

{In this section, we give two independent proofs: a rising-sun argument and
a Bellman function argument.}
\subsection{{A proof using the Riesz rising sun lemma}}

{In this subsection, we prove the one-sided John--Nirenberg inequality for
$\BLO$ functions using the Riesz rising sun lemma.  We first record some
elementary facts that will be used later.}

\begin{proposition}\label{prop:restriction}
If $J\subseteq I_0$, then
\[
 \norm{f}_{\BLO(J)}\le\norm{f}_{\BLO(I_0)}.
\]
\end{proposition}

\begin{proof}
Every subinterval of $J$ is also a subinterval of $I_0$, so the assertion
follows immediately from \eqref{eq:BLO-def}.
\end{proof}

\begin{proposition}\label{prop:affine}
Let $a>0$ and $b\in\R$.  Then
\[
 \norm{af+b}_{\BLO(I_0)}=a\norm{f}_{\BLO(I_0)}.
\]
In particular, addition of constants does not change the $\BLO$ seminorm.
\end{proposition}

\begin{proof}
For every interval $I\subseteq I_0$,
\[
 (af+b)_I-\essinf_I(af+b)
 =a\left(f_I-\essinf_I f\right).
\]
Taking the supremum over $I$ proves the claim.
\end{proof}

{We shall use the following one-dimensional form of the Riesz rising sun lemma.}

\begin{lemma}[{Riesz's rising sun lemma}]\label{lem:sunrise}
Let $I$ be a bounded interval, let $g\in L^1(I)$ be real-valued, and let
$a\in\R$ satisfy $g_I\le a$.  Then there is an at most countable family of
pairwise disjoint subintervals $\{I_j\}$ of $I$ such that
\begin{enumerate}
\item[(i)] $g_{I_j}=a$ for every $j$;
\item[(ii)] $g(x)\le a$ for almost every
$x\in I\setminus\bigcup_j I_j$.
\end{enumerate}
Consequently, up to a null set,
\[
 \{x\in I:g(x)>a\}\subseteq\bigcup_j I_j.
\]
\end{lemma}

\begin{proof}
{This is the classical rising sun lemma of Riesz~\cite{Riesz32}.  We include a
stopping-time construction  here.  If $g_I=a$, take $I_1=I$.  Otherwise $g_I<a$,
and we call $I$ active.  Given an active interval $J$, bisect it.  If both
halves have average less than $a$, retain both halves as active intervals.  If
one half has average at least $a$, the other has average less than $a$ because
$g_J<a$.  By continuity of the indefinite integral, move the bisection point
into the lower-average half until the larger of the two resulting intervals,
say $K$, satisfies $g_K=a$.  Put $K$ into the stopping family and retain the
complementary interval as active.  The retained interval has length at most
$|J|/2$ and average less than $a$.

Iterating this procedure produces pairwise disjoint stopping intervals
$\{I_j\}$ satisfying (i).  If a point does not belong to their union, then it
lies in a nested sequence of active intervals whose lengths tend to zero and
whose averages are less than $a$.  The Lebesgue differentiation theorem gives
$g(x)\le a$ at almost every such point, proving (ii).  The final inclusion is
immediate.}
\end{proof}

The normalization needed for the stopping argument is slightly more flexible
than $\essinf_I f=0$.

\begin{definition}\label{def:class}
Let $\mathcal A$ be the collection of all pairs $(I,g)$ such that $I$ is a
bounded interval and
\begin{equation}\label{eq:admissible}
 g\ge0\ \text{a.e. on }I,
 \qquad
 g_I\le1,
 \qquad
 \norm{g}_{\BLO(I)}\le1.
\end{equation}
For $\lambda\ge0$, define the distribution function as
\begin{equation}\label{eq:Phi}
 \Phi(\lambda)
 :=\sup_{(I,g)\in\mathcal A}
 \frac1{|I|}\abs{\{x\in I:g(x)>\lambda\}}.
\end{equation}
Clearly $0\le\Phi(\lambda)\le1$.
\end{definition}

{We first establish a recursive inequality for $\Phi$.}
\begin{lemma}\label{lem:recursion}
For every $\delta>0$ and every $\lambda\ge1+\delta$,
\begin{equation}\label{eq:recursion}
 \Phi(\lambda)
 \le\frac1{1+\delta}\Phi(\lambda-\delta).
\end{equation}
\end{lemma}

\begin{proof}
Fix $(I,g)\in\mathcal A$ and apply Lemma~\ref{lem:sunrise} to $g$ at the
height
\[
 a:=1+\delta.
\]
This is legitimate because $g_I\le1<a$.  Let $\{I_j\}$ be the resulting
stopping intervals.  Since $g\ge0$ and $g_{I_j}=a$, we have that
\begin{equation}\label{eq:length-stopping}
 a\sum_j|I_j|
 =\sum_j\int_{I_j}g
 \le\int_I g
 \le|I|.
\end{equation}
By the BLO condition on each $I_j$,
\[
 a-\essinf_{I_j}g
 =g_{I_j}-\essinf_{I_j}g
 \le1.
\]
Hence
\begin{equation}\label{eq:lower-on-stopping}
 \essinf_{I_j}g\ge a-1=\delta.
\end{equation}
On $I_j$, define
\[
 h_j=g-\delta.
\]
Then $h_j\ge0$ almost everywhere,
\[
 (h_j)_{I_j}=a-\delta=1,
 \qquad
 \norm{h_j}_{\BLO(I_j)}=\norm{g}_{\BLO(I_j)}\le1,
\]
so $(I_j,h_j)\in\mathcal A$.

Since $\lambda\ge a$, Lemma~\ref{lem:sunrise} implies, up to a null set,
\[
 \{g>\lambda\}\subseteq\bigcup_j I_j.
\]
Moreover,
\[
 \{x\in I_j:g(x)>\lambda\}
 =\{x\in I_j:h_j(x)>\lambda-\delta\}.
\]
It follows from the definition of $\Phi$ and \eqref{eq:length-stopping} that
\begin{align*}
 |\{g>\lambda\}|
 &\le\sum_j
 |\{x\in I_j:h_j(x)>\lambda-\delta\}|\\
 &\le\Phi(\lambda-\delta)\sum_j|I_j|\\
 &\le\frac{|I|}{1+\delta}\Phi(\lambda-\delta).
\end{align*}
Dividing by $|I|$ and taking the supremum over $(I,g)\in\mathcal A$,   we complete the proof.
\end{proof}

{With Lemma~\ref{lem:recursion} in hand, we obtain the following
distribution estimate.}
\begin{theorem}\label{thm:normalized}
For every $(I,g)\in\mathcal A$ and every $\lambda\ge0$,
\begin{equation}\label{eq:normalized-tail}
 \frac1{|I|}|\{x\in I:g(x)>\lambda\}|
 \le\min\{1,e^{1-\lambda}\}.
\end{equation}
\end{theorem}

\begin{proof}
We consider the case    $\lambda>1$ first.  Fix
$0<\delta<\lambda-1$ and set
\[
 n_\delta=\left\lfloor\frac{\lambda-1}{\delta}\right\rfloor.
\]
Then
\begin{equation}\label{eq:terminal-level}
 1\le\lambda-n_\delta\delta<1+\delta.
\end{equation}
For each $k=0,\ldots,n_\delta-1$, equation
\eqref{eq:terminal-level} yields
\[
 \lambda-k\delta
 =\lambda-n_\delta\delta+(n_\delta-k)\delta
 \ge1+\delta.
\]
Thus every one of the $n_\delta$ applications of
Lemma~\ref{lem:recursion} is legitimate, and the remaining argument of
$\Phi$ stays in its domain $[0,\infty)$.  Consequently,
\[
 \Phi(\lambda)
 \le(1+\delta)^{-n_\delta}
 \Phi(\lambda-n_\delta\delta)
 \le(1+\delta)^{-n_\delta},
\]
where the last inequality uses $\Phi\le1$.  Since
\[
 n_\delta\delta\longrightarrow\lambda-1
 \quad\text{and}\quad
 \frac{\log(1+\delta)}{\delta}\longrightarrow1
 \qquad(\delta\downarrow0),
\]
we obtain
\[
 \Phi(\lambda)\le e^{-(\lambda-1)}=e^{1-\lambda}.
\]
For $0\le\lambda\le1$, the estimate follows from $\Phi(\lambda)\le1$.
\end{proof}

{We can now prove the one-sided John--Nirenberg inequality for $\BLO$
functions using the Riesz rising sun lemma.}
\begin{theorem}\label{thm:upper}
Let $I_0$ be a bounded interval and let $f\in\BLO(I_0)$.  Then, for every
subinterval $I\subseteq I_0$ and every $\lambda>0$,
\[
 \frac1{|I|}
 \left|\left\{x\in I:
 f(x)-\essinf_I f>\lambda\right\}\right|
 \le
 \min\left\{1,
 \exp\left(1-\frac{\lambda}{\norm{f}_{\BLO(I_0)}}\right)
 \right\},
\]
with the usual interpretation when $\norm{f}_{\BLO(I_0)}=0$.
\end{theorem}

\begin{proof}
Write $B:=\norm{f}_{\BLO(I_0)}$.  If $B=0$, then
$f=\essinf_I f$ almost everywhere on every subinterval $I$, and the left-hand
side is zero. So, the inequality is established.

Assume $B>0$, fix $I\subseteq I_0$, and define
\[
 g=\frac{f-\essinf_I f}{B}\quad\text{on }I.
\]
Then $g\ge0$ almost everywhere.  Proposition~\ref{prop:restriction} gives
$\norm{g}_{\BLO(I)}\le1$, and
\[
 g_I=\frac{f_I-\essinf_I f}{B}\le1.
\]
Thus $(I,g)\in\mathcal A$.  Theorem~\ref{thm:normalized}, applied at the
level $\lambda/B$, gives the assertion.
\end{proof}

In particular, Theorem~\ref{thm:upper} proves that the pair
\[
 (c_1,c_2)=(e,1)
\]
is admissible in \eqref{eq:JN-BLO}.

\subsection{{A Bellman function proof}}\label{sec:bellman}

{In this subsection we give a second proof of Theorem~\ref{thm:upper},
independent of the Riesz rising sun lemma.  We use the Bellman function
method.  The argument also gives a refinement of \eqref{eq:main-bound} that
depends on the local mean excess.}

For $t>0$, define
\begin{equation*}\label{eq:bellman-q}
 q(t)=
 \begin{cases}
  t, & 0<t\le1,\\
  e^{t-1}, & t\ge1.
 \end{cases}
\end{equation*}
{Let
\[
 \overline{\mathcal D}
 =\{(z,y)\in\mathbb R^2:-1\le y\le z\le y+1\}.
\]
For a fixed $a\in\mathbb R$, consider the function
$U_a:\overline{\mathcal D}\to[0,1]$
given by
\begin{equation}\label{eq:bellman-candidate}
 U_a(z,y)=
 \begin{cases}
  1, & a\le y,\\[2mm]
  \displaystyle
  \min\left\{1,\frac{z-y}{q(a-y)}\right\}, & a>y.
 \end{cases}
\end{equation}
The variables $z$ and $y$ represent, respectively, the average and the
essential infimum of a function on a subinterval.  On the top edge
$z=y+1$, formula~\eqref{eq:bellman-candidate} is the continuous extension in
the $z$ variable of the same formula on $z<y+1$.}

\begin{lemma}[Bellman induction]\label{lem:bellman-induction}
Let $I$ be a bounded interval and let $h\in\BLO(I)$ satisfy
\[
 h_I=0,
 \qquad \norm{h}_{\BLO(I)}\le1.
\]
If $m=\essinf_I h$, then, for every $a\in\mathbb R$,
\begin{equation}\label{eq:bellman-tail}
 \frac1{|I|}\left|\{x\in I:h(x)>a\}\right|
 \le U_a(0,m).
\end{equation}
\end{lemma}

\begin{proof}
{We first verify the Bellman properties of $U_a$.  On the diagonal,
\begin{equation}\label{eq:bellman-obstacle}
 U_a(z,z)\ge \mathbf 1_{\{z>a\}}.
\end{equation}
For each fixed $y$, the function $z\mapsto U_a(z,y)$ is concave on
$[y,y+1]$: it is linear until it reaches the value $1$ and is constant
thereafter.

For each fixed $z$, the function $y\mapsto U_a(z,y)$ is nonincreasing on its
domain.  It is enough to check this where $a>y$ and $U_a(z,y)<1$.  If
$a-y\ge1$, then
\[
 U_a(z,y)=(z-y)e^{1-a+y}
\]
and hence
\[
 \frac{\partial U_a}{\partial y}(z,y)
 =e^{1-a+y}(z-y-1)\le0.
\]
If $0<a-y\le1$, then
\[
 U_a(z,y)=\frac{z-y}{a-y}.
\]
The inequality $U_a(z,y)<1$ implies $z<a$, and consequently
\[
 \frac{\partial U_a}{\partial y}(z,y)
 =\frac{z-a}{(a-y)^2}\le0.
\]
The formulas agree at $a-y=1$ and at the boundary of the region where the
minimum in \eqref{eq:bellman-candidate} equals $1$.  Thus
\begin{equation}\label{eq:bellman-monotone}
 y_1\le y_2\quad\Longrightarrow\quad
 U_a(z,y_1)\ge U_a(z,y_2).
\end{equation}
The same conclusions hold on the boundary of
$\overline{\mathcal D}$ by the one-sided limits in the relevant variable.}

{Assume first that $\norm{h}_{\BLO(I)}<1$.  Choose
$\eta\in(0,1/2)$ such that
$\norm{h}_{\BLO(I)}<1-\eta$.  We use the elementary balanced splitting lemma
for $\BLO$ functions: every subinterval $J$ can be divided into two adjacent
subintervals $J_-$ and $J_+$ such that
\begin{equation}\label{eq:bellman-balanced-split}
 \eta\le\frac{|J_\pm|}{|J|}\le1-\eta,
 \qquad
 h_{J_\pm}<\essinf_Jh+1.
\end{equation}
Here is the weighted argument.  Set
\[
 \alpha_\pm=\frac{|J_\pm|}{|J|},
 \qquad
 x_\pm=h_{J_\pm}-\essinf_Jh.
\]
For every dividing point,
\begin{equation}\label{eq:bellman-weighted-identity}
 \alpha_-x_-+\alpha_+x_+
 =h_J-\essinf_Jh<1-\eta.
\end{equation}
Start at the midpoint of $J$.  If $x_-<1$ and $x_+<1$, no adjustment is
needed.  Otherwise exactly one of the two values can be at least $1$ by
\eqref{eq:bellman-weighted-identity}.  Move the dividing point so as to shrink
the child having the smaller $x$-value.  The two averages, and hence
$x_-$ and $x_+$, depend continuously on the dividing point as long as both
children have positive length.  If their values cross, then at the crossing
they both equal $h_J-\essinf_Jh<1$.  If they do not cross before the smaller
child has relative length $\eta$, the larger child has relative length
$1-\eta$ and still has $x\ge1$.  Since both $x$-values are nonnegative, this
would make the left-hand side of
\eqref{eq:bellman-weighted-identity} at least $1-\eta$, a contradiction.
Thus the dividing point can be chosen while both relative lengths lie in
$[\eta,1-\eta]$ and both $x$-values are strictly smaller than $1$, which is
exactly \eqref{eq:bellman-balanced-split}.  This is the splitting lemma used
in the $\BLO$ Bellman method; see \cite[Lemma~2.1]{Osekowski}.}

For a subinterval $J$, write
\[
 z_J=h_J,
 \qquad y_J=\essinf_Jh.
\]
If $\alpha_\pm=|J_\pm|/|J|$, concavity in the first variable and
\eqref{eq:bellman-balanced-split} give
\[
 U_a(z_J,y_J)
 \ge
 \alpha_-U_a(z_{J_-},y_J)
 +\alpha_+U_a(z_{J_+},y_J).
\]
Since $y_{J_\pm}\ge y_J$, the monotonicity
\eqref{eq:bellman-monotone} yields
\begin{equation}\label{eq:bellman-dynamic}
 U_a(z_J,y_J)
 \ge
 \alpha_-U_a(z_{J_-},y_{J_-})
 +\alpha_+U_a(z_{J_+},y_{J_+}).
\end{equation}

Iterate this construction.  The balance condition in
\eqref{eq:bellman-balanced-split} implies that the lengths of the resulting
intervals tend uniformly to zero.  Let $h_n$ be the step function which is
equal to $h_J$ on every interval $J$ of the $n$th generation.  Iterating
\eqref{eq:bellman-dynamic}, using $y_J\le z_J$, and then applying
\eqref{eq:bellman-monotone} and \eqref{eq:bellman-obstacle}, we obtain
\[
 U_a(0,m)
 \ge \frac1{|I|}\int_I V_{\varepsilon,a}(h_n(x))\,dx,
\]
where
\[
 V_{\varepsilon,a}(t)=
 \begin{cases}
  0, & t\le a,\\
  (t-a)/\varepsilon, & a<t<a+\varepsilon,\\
  1, & t\ge a+\varepsilon.
 \end{cases}
\]
Indeed, $U_a(t,t)\ge V_{\varepsilon,a}(t)$ for every $t$.  By the Lebesgue
differentiation theorem, $h_n\to h$ almost everywhere.  Dominated convergence
and then $\varepsilon\downarrow0$ prove \eqref{eq:bellman-tail} when the
$\BLO$ seminorm is strictly smaller than $1$.

{Finally, suppose that $\norm{h}_{\BLO(I)}\le1$.  Since $h_I=0$, we have
$m=\essinf_Ih\ge-1$, so $(0,m)\in\overline{\mathcal D}$.  Apply the strict
case to
$\kappa h$ at the level $\kappa a$, where $0<\kappa<1$, and let
$\kappa\uparrow1$.  The left-hand distribution set is unchanged, and the
explicit formula~\eqref{eq:bellman-candidate} gives
\[
 U_{\kappa a}(0,\kappa m)\longrightarrow U_a(0,m).
\]
This proves \eqref{eq:bellman-tail}, including the top-boundary case $m=-1$.}
\end{proof}

The evaluation of the Bellman function at the initial state gives the promised
refinement.

\begin{theorem}[Bellman refinement]\label{thm:bellman-refinement}
Let $I_0$ be a bounded interval, let $f\in\BLO(I_0)$, and set
$B=\norm{f}_{\BLO(I_0)}>0$.  For a subinterval $I\subseteq I_0$, put
\[
 x_I=\frac{f_I-\essinf_I f}{B}\in[0,1],
 \qquad L=\frac{\lambda}{B}.
\]
Then
\begin{equation}\label{eq:bellman-refined-bound}
 \frac1{|I|}
 \left|\left\{x\in I:f(x)-\essinf_I f>\lambda\right\}\right|
 \le \mathcal B(x_I,L),
\end{equation}
where
\begin{equation*}\label{eq:bellman-value}
 \mathcal B(x,L)=\min\left\{1,\frac{x}{q(L)}\right\}
 =
 \begin{cases}
  1, & 0<L\le x,\\
  x/L, & x<L\le1,\\
  xe^{1-L}, & L\ge1.
 \end{cases}
\end{equation*}
At the common endpoints, the adjacent formulas agree.
\end{theorem}

\begin{proof}
Fix $I\subseteq I_0$ and write $s=\essinf_I f$.  On $I$, define
\[
 g=\frac{f-s}{B},
 \qquad h=g-x_I.
\]
Then
\[
 h_I=0,
 \qquad \essinf_Ih=-x_I,
 \qquad \norm{h}_{\BLO(I)}\le1.
\]
Moreover,
\[
 \{f-s>\lambda\}=\{h>L-x_I\}.
\]
Lemma~\ref{lem:bellman-induction}, with $a=L-x_I$, therefore gives
\[
 \frac1{|I|}\left|\{f-s>\lambda\}\right|
 \le U_{L-x_I}(0,-x_I)
 =\min\left\{1,\frac{x_I}{q(L)}\right\},
\]
which is \eqref{eq:bellman-refined-bound}.
\end{proof}

For $0<L\le1$, \eqref{eq:bellman-refined-bound} is at most $1$.  For
$L\ge1$, it is at most $e^{1-L}$ because $x_I\le1$.  Hence
\[
 \mathcal B(x_I,L)\le\min\{1,e^{1-L}\},
\]
and Theorem~\ref{thm:bellman-refinement} gives a second proof of
Theorem~\ref{thm:upper}.

\section{Sharpness}\label{sec:sharpness}
{In this section, we prove the sharpness assertions in
Theorem~\ref{thm:main}.  We first show that the exponential rate cannot
exceed $1$.}

\begin{lemma}\label{lem:logarithm}
Let $I_0=(0,1)$ and
\[
 f(x)=-\log x.
\]
Then
\[
 \norm{f}_{\BLO(I_0)}=1,
 \qquad
 \essinf_{I_0}f=0,
\]
and, for every $\lambda>0$,
\begin{equation}\label{eq:log-tail}
 |\{x\in(0,1):f(x)>\lambda\}|=e^{-\lambda}.
\end{equation}
\end{lemma}

\begin{proof}
For $0<a<b\le1$,
\[
 \frac1{b-a}\int_a^b(-\log x)\,dx
 =1-\frac{b\log b-a\log a}{b-a}.
\]
Since $f$ is decreasing,
$\essinf_{(a,b)}f=-\log b$, and therefore
\begin{equation*}\label{eq:log-blo-computation}
 f_{(a,b)}-\essinf_{(a,b)}f
 =1-\frac{a\log(b/a)}{b-a}
 \le1.
\end{equation*}
The same formula holds at $a=0$ by taking a limit, and then the value is
$1$.  Hence $\norm{f}_{\BLO((0,1))}=1$.  Finally,
$-\log x>\lambda$ if and only if $x<e^{-\lambda}$, proving
\eqref{eq:log-tail}.
\end{proof}
{This yields the sharp exponential rate $c_2^*$.}
\begin{corollary}\label{cor:rate-sharp}
If \eqref{eq:JN-BLO} holds with a finite constant $c_1$, then $c_2\le1$.
Thus $c_2^*=1$.
\end{corollary}

\begin{proof}
Apply \eqref{eq:JN-BLO} to the function in Lemma~\ref{lem:logarithm}.  We get
\[
 e^{-\lambda}\le c_1e^{-c_2\lambda}
 \qquad(\lambda>0).
\]
Equivalently, $c_1\ge e^{(c_2-1)\lambda}$ for every $\lambda>0$.
If $c_2>1$, the right-hand side tends to infinity as
$\lambda\to\infty$, a contradiction.  The reverse inequality
$c_2^*\ge1$ follows from Theorem~\ref{thm:upper}.
\end{proof}

We next determine the smallest prefactor at the sharp rate.

\begin{lemma}\label{lem:indicator}
Fix $0<p<1$ and define on $I_0=(0,1)$
\[
 h_p=\mathbf 1_{(0,p)}.
\]
Then
\[
 \norm{h_p}_{\BLO(I_0)}=1,
 \qquad
 \essinf_{I_0}h_p=0.
\]
Moreover, for every $0<\lambda<1$,
\[
 |\{x\in(0,1):h_p(x)>\lambda\}|=p.
\]
\end{lemma}

\begin{proof}
On any interval $J\subseteq(0,1)$, the function $h_p$ takes only the values
$0$ and $1$. So
\[
 (h_p)_J-\essinf_Jh_p\le1.
\]
{On intervals of the form $(a,p+\varepsilon)$ with
$0<a<p$ and $0<\varepsilon<1-p$, the essential infimum is zero and the average
can be made arbitrarily close to $1$ by taking $a\downarrow0$ and
$\varepsilon\downarrow0$.  Hence the supremum equals $1$.}
The distribution statement is immediate.
\end{proof}

\begin{corollary}\label{cor:prefactor-sharp}
At the sharp exponential rate $c_2^*=1$, every admissible prefactor satisfies
$c_1\ge e$.  Therefore $c_1^*=e$.
\end{corollary}

\begin{proof}
Apply \eqref{eq:JN-BLO} with $c_2=1$ to $h_p$.  For every
$0<\lambda<1$,
\[
 p\le c_1e^{-\lambda}.
\]
Letting $\lambda\uparrow1$ gives $c_1\ge pe$, and then letting $p\uparrow1$
gives $c_1\ge e$.  The upper bound $c_1^*\le e$ follows from
Theorem~\ref{thm:upper}.
\end{proof}

Corollaries~\ref{cor:rate-sharp} and~\ref{cor:prefactor-sharp} complete the
proof of Theorem~\ref{thm:main}.

\begin{remark}\label{rem:failed-alpha}
The logarithmic example also provides a useful consistency check for any
proposed one-sided inequality.  Since $-\log x$ has $\BLO$ seminorm $1$ and exact
tail $e^{-\lambda}$ on $(0,1)$, an exponential rate strictly larger than $1$
is impossible, regardless of the prefactor.
\end{remark}

\section{Applications}\label{sec:applications}

{In this section we record several direct consequences of the sharp
distribution estimate.  Besides exponential integrability and quantitative
$L^p$ bounds, the estimate provides a connection between $\BLO$ and the
Muckenhoupt class $A_1$.  Logarithms of $A_1$ weights give a classical
characterization of $\BLO$; see \cite{CR80}.  Sharp one-dimensional
$\BLO_p$--$\BLO$ and exponential $A_1$ estimates were developed by
Passarelli di Napoli~\cite{Passarelli96}, while Popoli~\cite{Popoli25}
identified the exact small-seminorm threshold for exponentiation into
$A_1$.  Accordingly, the results below are stated as corollaries of
Theorem~\ref{thm:main}; apart from the explicitly proved exponent range and
large-$p$ asymptotic, no claim of new optimal constants is made.}

\subsection{Exponential integrability and Muckenhoupt weights}

For a positive locally integrable function $w$ on $I_0$, define its localized
$A_1$ constant by
\[
 [w]_{A_1(I_0)}
 =\sup_{I\subseteq I_0}
 \frac{w_I}{\essinf_I w}.
\]

\begin{corollary}\label{prop:exp-integrability}
Let $f\in\BLO(I_0)$ and set $B=\norm{f}_{\BLO(I_0)}>0$.  For every
$0<\alpha<1$ and every interval $I\subseteq I_0$,
\begin{equation}\label{eq:exp-integrability}
 \frac1{|I|}\int_I
 \exp\left(\alpha\frac{f(x)-\essinf_I f}{B}\right)\,dx
 \le \frac{e^\alpha}{1-\alpha}.
\end{equation}
Consequently, the weight
\[
 w_\alpha=\exp\left(\frac{\alpha f}{B}\right)
\]
belongs to $A_1(I_0)$ and satisfies
\begin{equation}\label{eq:A1-constant}
 [w_\alpha]_{A_1(I_0)}\le \frac{e^\alpha}{1-\alpha}.
\end{equation}
The range $0<\alpha<1$ is optimal uniformly over the unit ball of
$\BLO(I_0)$.
\end{corollary}

\begin{proof}
Fix $I\subseteq I_0$, put $m_I=\essinf_I f$, and set
\[
 X=\frac{f-m_I}{B}.
\]
Then $X\ge0$, and Theorem~\ref{thm:upper} gives
\[
 \frac1{|I|}\abs{\{x\in I:X(x)>t\}}
 \le \min\{1,e^{1-t}\},\qquad t>0.
\]
The layer-cake formula and Tonelli's theorem therefore yield
\begin{align*}
 \frac1{|I|}\int_I e^{\alpha X(x)}\,dx
 &=1+\alpha\int_0^\infty e^{\alpha t}
   \frac{\abs{\{x\in I:X(x)>t\}}}{|I|}\,dt\\
 &\le 1+\alpha\int_0^1e^{\alpha t}\,dt
       +\alpha e\int_1^\infty e^{-(1-\alpha)t}\,dt\\
 &=\frac{e^\alpha}{1-\alpha}.
\end{align*}
This proves \eqref{eq:exp-integrability}.  Since the exponential function is
increasing,
\[
 \essinf_I w_\alpha
 =\exp\left(\frac{\alpha m_I}{B}\right).
\]
Multiplying \eqref{eq:exp-integrability} by this quantity and then taking the
supremum over $I\subseteq I_0$ proves \eqref{eq:A1-constant}.

To see that the exponent range cannot be enlarged, take
$I_0=(0,1)$ and $f(x)=-\log x$.  Lemma~\ref{lem:logarithm} gives
$\norm{f}_{\BLO(I_0)}=1$, whereas
\[
 \int_0^1 e^{\alpha f(x)}\,dx
 =\int_0^1x^{-\alpha}\,dx
\]
is infinite for $\alpha\ge1$.  Thus neither uniform exponential integrability nor a uniform \(A_1\) conclusion is possible at or beyond
$\alpha=1$.
\end{proof}

{Corollary~\ref{prop:exp-integrability} records the consequence of
Theorem~\ref{thm:main} for the $\BLO$--$A_1$ correspondence: the sharp
exponential rate yields the uniform threshold for exponentiating normalized
$\BLO$ functions.  The estimate \eqref{eq:A1-constant} is the bound obtained
from the distribution estimate; in view of the sharper literature cited
above, we do not claim that its prefactor is optimal or new.}

\subsection{Quantitative moment estimates}

Let
\[
 \Gamma(s,a)=\int_a^\infty t^{s-1}e^{-t}\,dt
\]
denote the upper incomplete gamma function.

\begin{corollary}\label{prop:Lp-application}
Let $f\in\BLO(I_0)$ and $B=\norm{f}_{\BLO(I_0)}$.  For every $p>0$ and
every interval $I\subseteq I_0$,
\begin{equation}\label{eq:Lp-application}
 \left(\frac1{|I|}\int_I
 \bigl(f(x)-\essinf_I f\bigr)^p\,dx\right)^{1/p}
 \le \bigl(e\Gamma(p+1,1)\bigr)^{1/p}B.
\end{equation}
Moreover,
\begin{equation}\label{eq:Lp-asymptotic}
 \bigl(e\Gamma(p+1,1)\bigr)^{1/p}\sim\frac{p}{e}
 \qquad (p\to\infty),
\end{equation}
and the coefficient $1/e$ in this asymptotic growth is optimal.
\end{corollary}

\begin{proof}
The assertion is immediate when $B=0$.  Suppose that $B>0$, fix
$I\subseteq I_0$, and again put
$X=(f-\essinf_I f)/B$.  Integrating the distribution function gives
\begin{align*}
 \frac1{|I|}\int_I X(x)^p\,dx
 &=p\int_0^\infty t^{p-1}
   \frac{\abs{\{x\in I:X(x)>t\}}}{|I|}\,dt\\
 &\le p\int_0^1t^{p-1}\,dt
       +ep\int_1^\infty t^{p-1}e^{-t}\,dt\\
 &=1+ep\Gamma(p,1)
 =e\Gamma(p+1,1),
\end{align*}
where the last equality follows by integration by parts.  This proves
\eqref{eq:Lp-application}.  Formula \eqref{eq:Lp-asymptotic} follows from
Stirling's formula.  Finally, for the logarithmic example of
Lemma~\ref{lem:logarithm},
\[
 \left(\int_0^1(-\log x)^p\,dx\right)^{1/p}
 =\Gamma(p+1)^{1/p}\sim\frac{p}{e}.
\]
Hence the asymptotic coefficient cannot be decreased.
\end{proof}

{Thus Theorem~\ref{thm:main} controls all positive moments of the
one-sided oscillation.  Corollary~\ref{prop:Lp-application} is included to
record this consequence; its constant is not asserted to be optimal for each
fixed $p$, although its large-$p$ growth is sharp.}

\subsection{{Outputs of global BMO-to-BLO estimates}}

{The sharp estimate also applies to mappings whose endpoint range is
$\BLO$.  To avoid imposing an unproved localization property, we formulate
the observation using global seminorms.  Suppose that a mapping $T$ on
functions on $\mathbb R$ satisfies
\[
 \norm{Th}_{\BLO(\mathbb R)}
 \le C_T\norm{h}_{\BMO(\mathbb R)}.
\]
Since the localized $\BLO$ seminorm on every bounded interval $I_0$ is at
most the global seminorm, Theorem~\ref{thm:upper} gives, for every bounded
interval $I\subset\mathbb R$ and every $\lambda>0$,
\begin{equation}\label{eq:operator-application}
 \frac1{|I|}\abs{\left\{x\in I:
 Th(x)-\essinf_I Th>\lambda\right\}}
 \le
 \min\left\{1,
 \exp\left(1-
 \frac{\lambda}{C_T\norm{h}_{\BMO(\mathbb R)}}\right)\right\},
\end{equation}
with the usual interpretation when the denominator is zero.  Thus a global
$\BMO$-to-$\BLO$ estimate yields one-sided exponential concentration on
every bounded interval.  Bennett's theorem for nonabsolute maximal operator \(Mh(x)=\sup_{I\ni x}h_I\) is a classical example, subject to the usual finiteness hypothesis;
see \cite{Bennett82}.

The scale of the endpoint estimate must be retained when this observation is
applied.  In the Littlewood--Paley setting of \cite{LYZ}, Corollary~2.6 gives
an estimate of the form
\[
 \norm{g(h)}_{\BLO(\mathbb R)}
 \le C\norm{h}_{\BMO(\mathbb R)},
\]
so \eqref{eq:operator-application} applies to $T(h)=g(h)$ with the linear
scale $C\norm{h}_{\BMO}$.  By contrast, Theorem~2.4 of the same reference
gives
\[
 \norm{[g(h)]^2}_{\BLO(\mathbb R)}
 \le C\norm{h}_{\BMO(\mathbb R)}^2.
\]
For the output $[g(h)]^2$, the denominator in the exponential tail is
therefore $C\norm{h}_{\BMO}^2$, not $C\norm{h}_{\BMO}$.  Results for
differential transforms can be treated in the same way after inserting the
corresponding global $\BMO$-to-$\BLO$ bound; see \cite{YZ26}.}


\begingroup

\begin{thebibliography}{99}

\bibitem{Bennett82}
C.~Bennett,
\emph{Another characterization of $\BLO$},
Proc. Amer. Math. Soc. \textbf{85} (1982), no.~4, 552--556.

\bibitem{CR80}
R.~R.~Coifman and R.~Rochberg,
\emph{Another characterization of $\BMO$},
Proc.\ Amer.\ Math.\ Soc.\ \textbf{79} (1980), no.~2, 249--254.



\bibitem{YZ26}
W. Hu, K. Wu, D. Yang, and C. Zhang,
\emph{Boundedness of variation, oscillation, and differential transform on $\BMO$ space},
Michigan Math. J. \textbf{76} (2026), 339--364.

\bibitem{JN61}
F.~John and L.~Nirenberg,
\emph{On functions of bounded mean oscillation},
Comm.\ Pure Appl.\ Math.\ \textbf{14} (1961), 415--426.

\bibitem{Kore}
A.~A.~Korenovskii,
\emph{The connection between mean oscillations and exact exponents of summability of functions},
Math.\ USSR-Sb.\ \textbf{71} (1992), no.~2, 561--567;
translated from Mat.\ Sb.\ \textbf{181} (1990), no.~12, 1721--1727.

\bibitem{KLS05}
A.~A.~Korenovskyy, A.~K.~Lerner, and A.~M.~Stokolos,
\emph{On a multidimensional form of F. Riesz's ``rising sun'' lemma},
Proc. Amer. Math. Soc. \textbf{133} (2005), no.~5, 1437--1440.

\bibitem{Lerner13}
A.~K.~Lerner,
\emph{The John--Nirenberg inequality with sharp constants},
C.~R.\ Math.\ Acad.\ Sci.\ Paris \textbf{351} (2013), no.~11--12, 463--466.

\bibitem{LYZ}
S.~Liang, D.~Yang and C.~Zhang,
\emph{New characterizations of $\BLO$ spaces by heat semigroups and applications},
arXiv:2602.00479.

\bibitem{Osekowski}
A.~Os\c{e}kowski,
\emph{Sharp estimates for functions of bounded lower oscillation},
Bull.\ Austral.\ Math.\ Soc.\ \textbf{87} (2013), 68--81.

\bibitem{Passarelli96}
{A.~Passarelli di Napoli,
\emph{Sharp inequalities for $\BLO$ in one dimension},
Atti Sem. Mat. Fis. Univ. Modena \textbf{44} (1996), no.~2, 465--477.}

\bibitem{Popoli25}
{A.~Popoli,
\emph{Functions with small $\BMO$ norm},
Proc. Roy. Soc. Edinburgh Sect. A (2025), published online,
\href{https://doi.org/10.1017/prm.2024.141}{doi:10.1017/prm.2024.141}.}

\bibitem{Riesz32}
F.~Riesz,
\emph{Sur un th\'eor\`eme de maximum de MM. Hardy et Littlewood},
J. London Math. Soc. \textbf{7} (1932), no.~1, 10--13.

\bibitem{SV11}
L.~Slavin and V.~Vasyunin,
\emph{Sharp results in the integral-form John--Nirenberg inequality},
Trans.\ Amer.\ Math.\ Soc.\ \textbf{363} (2011), 4135--4169.

\bibitem{VV14}
V.~Vasyunin and A.~Volberg,
\emph{Sharp constants in the classical weak form of the John--Nirenberg inequality},
Proc.\ Lond.\ Math.\ Soc.\ \textbf{108} (2014), no.~6, 1417--1434.

\bibitem{WZT18}
D.~Wang, J.~Zhou, and Z.~Teng,
\emph{Some characterizations of $\BLO$ space},
Math.\ Nachr.\ \textbf{291} (2018), no.~11--12, 1908--1918.

\end{thebibliography}
\endgroup
\end{document}